\documentclass{amsart}[12pt]

\usepackage{amsmath,amsfonts,amssymb}
\usepackage{bm}
\usepackage{comment}
\newcommand{\la}{\lambda}
\newcommand{\be}{\beta}

\newcommand{\de}{\delta}

\newcommand{\e}{\varepsilon}

\newcommand{\BQ}{\mathbb{Q}}

\newcommand{\BN}{\mathbb{N}}

\newcommand{\ov}{\overline}

\newcommand{\Erdos}{Erd{\H o}s}

\newcommand{\E}{\mathcal E}

\newtheorem{lemma}{Lemma}[section]

\newtheorem{prop}[lemma]{Proposition}
\newtheorem{thm}[lemma]{Theorem}
\newtheorem{cor}[lemma]{Corollary}
\theoremstyle{definition}

\theoremstyle{remark}
\newtheorem{rmk}[lemma]{Remark}

\numberwithin{equation}{section}
\numberwithin{table}{section}

\begin{document}
\title[A lower bound for the dimension of Bernoulli convolutions]
{A lower bound for the dimension of Bernoulli convolutions}
%\\
%\today}
\author{Kevin G. Hare}
\email{kghare@uwaterloo.ca}
\address{Department of Pure Mathematics, University of Waterloo, Waterloo, Ontario, Canada N2L 3G1}
\thanks{Research of K.G. Hare was supported by NSERC Grant RGPIN-2014-03154}
\thanks{Computational support provided in part by the Canadian Foundation for Innovation,
    and the Ontario Research Fund.}
\author{Nikita Sidorov}
\email{sidorov@manchester.ac.uk}
\address{School of Mathematics, The University of Manchester, Oxford Road, Manchester M13 9PL, United Kingdom}
\date{\today}

\begin{abstract}
Let $\be\in(1,2)$ and let $H_\be$ denote Garsia's entropy for the Bernoulli convolution
$\mu_\be$ associated with $\be$.
In the present paper we show that $H_\be>0.82$ for all $\be \in (1, 2)$ and improve this bound
for certain ranges. Combined with recent
results by Hochman and Breuillard-Varj\'u, this yields $\dim (\mu_\be)\ge0.82$ for all $\be\in(1,2)$.

In addition, we show that if an algebraic $\be$ is such that
    $[\mathbb{Q}(\be): \mathbb{Q}(\be^k)] = k$ for some $k \geq 2$ then
    $\dim(\mu_\be)=1$. Such is, for instance, any root of a Pisot
    number which is not a Pisot number itself.
\end{abstract}

\maketitle

\section{Introduction and summary}
\label{sec:intro}

Bernoulli convolutions are a class of measures on the real line with a compact support
which are closely related to $\beta$-expansions introduced
by R\'enyi \cite{Re} and first studied by R\'enyi and by Parry \cite{Pa, Re}.
Let us recall the basic definitions.

Suppose $\beta$ is a real number greater than~1. A {\em $\beta$-expansion}
of the real number $x \in [0,1]$ is an infinite sequence of integers
$(a_1, a_2, a_3, \dots)$ such that $x = \sum_{n \ge 1}a_n \beta^{-n}$.
%The reader is referred to Lothaire, \cite[Chapter 7]{Lo} for more on these topics.
For the purposes of this paper, we assume that $1 < \beta < 2$ and $a_i \in \{0, 1\}$.

Let $\mu_\be$ denote the {\em Bernoulli convolution} parameterized by $\be$ on $I_\be:=[0,1/(\be-1)]$, i.e.,
\[
\mu_\be(E)=\mathbb P\left\{(a_1,a_2,\dots)\in\{0,1\}^\BN : \sum_{k=1}^\infty a_k\be^{-k}\in E\right\}
\]
for any Borel set $E\subseteq I_\be$, where $\mathbb P$ is the product measure on $\{0,1\}^\BN$
with $\mathbb P(a_1=0)=\mathbb P(a_1=1)=1/2$. Since $\be<2$, it is obvious that $\text{supp}\,(\mu_\be)=I_\be$.

Bernoulli convolutions have been studied since the 1930s (see, e.g., \cite{Varju-survey} and references therein).
An important property of $\mu_\be$ is the fact that it is either absolutely continuous (in which case
it is actually equivalent to the Lebesgue measure -- see \cite{MS}) or purely singular.

Recall that a number $\be>1$ is called a {\em Pisot number} if it is an algebraic integer whose other Galois conjugates
are less than~1 in modulus. \Erdos\ \cite{Erdos39} showed that if $\be$ is a Pisot number,
then $\mu_\be$ is singular. Garsia \cite{Garsia62} introduced a new class of algebraic integers
-- now referred to as {\em Garsia numbers} -- and proved that $\mu_\be$ is absolutely continuous if $\be$ is
a Garsia number. More recently, Varj\'u \cite{Varju-Bernoulli} found a new class of $\be$ with this
property, though no explicit examples are given. Solomyak \cite{Sol} proved that for Lebesgue-a.e.
$\be\in(1,2)$ the Bernoulli convolution is absolutely continuous.

It is also known that $\mu_\be$ is exact-dimensional (which is a special case of a general result in \cite{Hochman}).
Namely, there is a number $\alpha$ such that
\[
\lim_{r\to 0}\frac{\log \mu_\be(x-r,x+r)}{\log r}=\alpha
\]
for $\mu_\be$-almost every $x$. We will call this number $\alpha$ {\em the dimension of $\mu_\be$} and
denote by $\dim(\mu_\be)$. In particular, the exact-dimensionality implies $\dim_H(\mu_\be)=\dim(\mu_\be)$.
Clearly, if $\mu_\be$ is absolutely continuous, then $\dim(\mu_\be)=1$. Whether the converse is true for
this family, remains unknown.

Garsia \cite{Garsia63} introduced the following useful quantity associated with a Bernoulli convolution. Namely, put
\[
D_n(\be)=\left\{x\in I_\be: x=\sum_{k=1}^n a_k\be^{-k}\,\, \text{with}\,\, a_k\in\{0,1\}\right\}
\]
and for $x\in D_n(\be)$,
\begin{equation}\label{eq:pnx}
p_n(x)=\#\left\{(a_1,\dots,a_n)\in\{0,1\}^n : x=\sum_{k=1}^n a_k \be^{-k}\right\}.
\end{equation}
Let $\mu_\be^{(n)}$ denote the measure corresponding to the distribution $\bigl\{\frac{p_n(x)}{2^n} : x\in D_n(\be)\bigr\}$
and let $H_n(\be)=H(\mu_\be^{(n)})$, the entropy of the corresponding distribution, i.e.,
\[
H_n(\be)=-\sum_{x\in D_n(\be)} \frac{p_n(x)}{2^n}\log \frac{p_n(x)}{2^n}.
\]
Finally, put
\[
H_\be=\lim_{n\to\infty}\frac{H_n(\be)}{n\log\be}
\]
(it was shown in \cite{Garsia63} that the limit always exists, in view of $\{H_n(\be)\}_{n\ge1}$ being
subadditive). The value $H_\be$ is called {\em Garsia's entropy}.

Obviously, if $\be$ is transcendental or algebraic but not satisfying an algebraic equation with coefficients $\{-1,0,1\}$ (i.e.,
it is not of {\em height one}), then all the sums $\sum_{k=1}^n a_k\be^{-k}$ are distinct, whence $p_n(x)=1$
for any $x\in D_n(\be)$, and $H_\be=\log2/\log\be>1$.

However, if $\be$ is Pisot, then it was shown in \cite{Garsia63, PU} that $H_\be=\dim(\mu_\be)<1$
-- which means in particular that $\be$ is of height one. Furthermore, Garsia also proved that if
$H_\be<1$, then $\mu_\be$ is singular.

Alexander and Zagier in \cite{AlexanderZagier91} managed to evaluate $H_\be$ for the golden ratio
$\be=\tau$ with an astonishing accuracy. It turned out that $H_\tau$ is close to 1 -- in fact $H_\tau\approx0.9957$.
Grabner, Kirschenhofer and Tichy \cite{GrabnerKirschenhoferTichy02} extended this method to the multinacci numbers,
which are the positive real roots of $x^m=x^{m-1}+x^{m-2}+\dots+x+1$. No other values of $H_\be$ correct up to
at least two decimal places are known to date.

In our previous paper on this subject \cite{HS2010} we showed that if $\be$ is a Pisot number, then $H_\be>0.81$ and
improved this bound for various ranges. The main purpose of the present paper is to extend this result to
all $\be$ with an improved lower bound:
\begin{thm}\label{thm:082}
We have
\[
H_\be>0.82
\]
for all $\be\in(1,2)$.
\end{thm}
We also show that $H_\be>1$ for certain classes
of algebraic non-Pisot $\be$ of height~one (Theorem~\ref{thm:pisot-entropy}).

The main reason we have decided to return to this topic is a recent breakthrough made by Hochman \cite{Hochman} (see
also \cite[Theorem~19]{BV} for a detailed explanation).

\begin{thm} [Hochman, 2014]\label{thm:hochman}
If $\be\in(1,2)$ is algebraic, then
\[
\dim(\mu_\be)=\min\{H_\be,1\}.
\]
\end{thm}
Essentially, this remarkable theorem says that if we have even a weak separation
of the elements of $D_n(\be)$, then the topological quantity, $\dim(\mu_\be)$,
coincides with the combinatorial one, $H_\be$. If $\be$ is Pisot, then this is
relatively straightforward (and known since \cite{PU}); for all other algebraic $\be$
this fact is highly non-trivial.

Another important result was proved recently in \cite[Corollary~6]{BV2}:

\begin{thm}[Breuillard and Varj\'u, 2016] 
We have
\[
\min_{\be\in(1,2)}\dim \mu_\be = \inf_{\be\in(1,2)\cap\overline{\BQ}}\dim \mu_\be,
\]
where $\ov\BQ$ stands for the set of algebraic numbers. 
\end{thm}

Combined with Theorem~\ref{thm:082}, these yield the main theorem of the present paper.

\begin{thm}\label{thm:main}
For all $\be\in(1,2)$ we have $\dim(\mu_\be)\ge 0.82$.
\end{thm}

%\begin{rmk}If $\be$ is not of height one, then $H_\be>1$, so $\dim(\mu_\be)=1$. Thus,
%Theorem~\ref{thm:main}'s real application are the $\be$'s of height one.
%\end{rmk}

\section{Theoretical results}

\begin{prop}
\label{prop:Hbetak}
For all natural $k\ge2$, we have $H_\beta \ge H_{\beta^k}$.
\end{prop}

\begin{proof}
We have $D_n(\be^k)\subset D_{nk}(\be)$. Furthermore,
\[
\sum_{j=1}^{kn} a_j\be^{-j}=\sum_{j=1}^n a_{kj}\be^{-kj} + \text{another sum},
\]
the sums being independent. Hence
\[
\mu_\be^{(kn)}=\mu_{\be^k}^{(n)}*\nu
\]
for some probability measure $\nu$. Since the entropy of a convolution is always greater
than or equal to the entropy of either measure involved, we have
\[
H_{kn}(\be)\ge H_n(\be^k).
\]
Hence
\[
\frac{H_{kn}(\be)}{kn\log\be} \ge \frac{H_n(\be^k)}{n\log\be^k}.
\]
We now conclude the proof by passing to the limit as $n\to\infty$.
\end{proof}

The following claim will be used in Section~\ref{sec:alg} which deals with the numerical
results.

\begin{cor}\label{cor:1526}
If $\be\in(\sqrt2,1.526]$, then $H_\be>\log2/\log(1.526^2)>0.82$.
\end{cor}
\begin{proof}Apply Proposition~\ref{prop:Hbetak} with $k=2$.
\end{proof}

The next corollary is much weaker than Theorem~\ref{thm:main} proved in Section~\ref{sec:alg}
but is still worth mentioning because it is straightforward and does not require any
computations while still being applicable.

\begin{cor}\label{cor:greater-than-a-half}
We have $H_\be>\frac12$ for all $\be\in(1,2)$.
\end{cor}
\begin{proof}If $\be\in(\sqrt2,2)$, then $H_\be\ge H_{\be^2}=\frac{\log2}{2\log\be}>\frac12$.
Further $H_{\sqrt2}=2$.
Finally for $\be\in(1,\sqrt2)$ we simply take $k$ such that $\sqrt{2} \leq \be^k < 2$ and apply
Proposition~\ref{prop:Hbetak}.
\end{proof}

\begin{rmk}Fix $k\ge2$; then it follows that for any $\e>0$ there exists $\de>0$ such that
$H_\be>1-\e$ whenever $\be\in(2^{1/k}, 2^{1/k}+\de)$. In particular, this allows us to improve
some of our our bounds from \cite{HS2010}. Namely, let $\be_3\approx1.4433$ and $\be_4\approx1.4656$
denote the third and fourth smallest Pisot numbers. Then $H_{\be_3}>0.9445, H_{\be_4}>0.9066$.
\end{rmk}

\begin{prop}
\label{prop:notinfield}
If $[\mathbb{Q}(\beta^{1/k}): \mathbb{Q}(\beta)] = k$ for some $k \geq 2$,
    then $H_{\be^{1/k}}=kH_{\be}$.
\end{prop}
\begin{proof}
For convenience, let $\be_0 = \be^{1/k}$.
Assume first $k=2$. Put $X=\sum_1^n a_{2j}\be_0^{-2j},\ Y=\sum_1^n a_{2j-1}\be_0^{-2j+1}$.
We have
\[
H(X+Y)=H(X)+H(Y)-H(Y|X+Y).
\]
It is well known that $H(Y|X+Y)=0$ if and only if the value of $Y$ is completely determined by
the value of $X+Y$. This is indeed the case here: if $X+Y=X'+Y'$, then $X=X'$ and $Y=Y'$, in view of
$\be_0\notin\mathbb Q(\be_0^2)$. Hence $H(X+Y)=H(X)+H(Y)$. We have, in our usual notation,
$H(X+Y)=H_{2n}(\be_0),\ H(X)=H(Y)=H_n(\be_0^2)$. Therefore, $H_{2n}(\be_0)=2H_n(\be_0^2)$, and the claim
follows from dividing by $2n\log\be_0$ and passing to the limit.

Assume now $k\ge3$ and proceed by induction. We have
\[
H\left(\sum_{j=1}^k X_j\right)=H\left(\sum_{j=1}^{k-1} X_j\right)+H(X_k)-H(X_k | X_1+\dots +X_k).
\]
Since the value of $X_1+\dots +X_j$ determines the value of $X_i,\ 1\le i\le j$, the last summand is zero,
and we can apply the inductive step to $H\left(\sum_{j=1}^{k-1} X_j\right)$. Consequently,
$H(X_1+\dots+X_k)=H(X_1)+\dots H(X_k)=kH(X_1)$, which yields the claim.
\end{proof}

Although it is possible for $[\mathbb{Q}(\beta^{1/k}): \mathbb{Q}(\beta)] \neq
    k$, this does not happen often.
See for example \cite[Theorem~9.1, p.~297]{Lang93}, which we
    rewrite here in our notation.

\begin{thm}[Lang \cite{Lang93}]
\label{thm:Lang}
Let $K$ be a field and $k$ an integer $\geq 2$.
Let $\beta \in K$, $\beta \neq 0$.
Assume that for all prime numbers $p$ such that
    $p | k$ there does not exist an $\alpha \in K$
    such that $\beta = \alpha^p$.
Further if $4 | k$ assume that there does not exist an $\alpha \in K$
    such that $\beta = -4 \alpha^4$.
Then $x^k - \beta$ is irreducible in $K[x]$,
    that is $[K(\beta^{1/k}): K] = k$.
\end{thm}

\begin{thm}\label{thm:pisot-entropy}
Let $k\ge2$ and $\beta \in (1,2)$.
Assume for all $p | k$ that
    there does not exist an $\alpha \in \mathbb{Q}(\beta)$ such
    that $\beta = \alpha^p$.
Then $H_{\be^{1/k}}\geq 1$ and consequently, $\dim(\mu_{\be^{1/k}})=1$.
\end{thm}

\begin{proof}
By assumption, for all $p | k$
    there does not exists an $\alpha \in \mathbb{Q}(\beta)$ such
    that $\beta = \alpha^p$.
Further as $\beta > 1$ we easily see that
    $\beta > 0 > -4 \alpha^4$ for all $\alpha \in \mathbb{Q}(\beta)$.
Hence $x^k - \beta$ is irreducible in $\mathbb{Q}(\beta)[x]$.
Hence $[\mathbb{Q}(\beta^{1/k}): \mathbb{Q}(\beta)] = k$.
Hence $H_{\beta^{1/k}} \geq k H_\beta$.
By Corollary~\ref{cor:greater-than-a-half},
$H_\beta > \frac12$, whence $H_{\beta^{1/k}} \geq k \cdot \frac12 \geq 1$, as required.
\end{proof}

\begin{rmk}It is easy to see with a polynomial with coefficients in $\{-1,0,1\}$
cannot have a root of modulus greater than~2. Hence $\be>2$ cannot be of height~one.
Consequently, although Theorem~\ref{thm:pisot-entropy} is valid for $\be\in[2,2^k)$,
the claim is trivial.
\end{rmk}

\begin{cor}
Let $k \geq 2$ and $\beta \in (1,2)$ a Pisot number.
Assume further that $\beta^{1/k}$ is not a Pisot number.
Then $H_{\beta^{1/k}} \geq 1$ and $\dim(\mu_{\beta^{1/k}}) = 1$.
\end{cor}

\begin{proof}
Let $\alpha = \beta^{1/k}$.
Choose $n | k$ such that $\alpha^n$ is Pisot, and for all $n' | n$, $n' < n$
    we have that $\alpha^{n'}$ is not Pisot.
Let $\beta_0 = \alpha^n$.
We claim that $\beta_0$ satisfies the conditions of Theorem
    \ref{thm:pisot-entropy}, from which the result will follow.

Let $p | n$ and assume that $\beta_0 = \alpha_0^p$ for some
    $\alpha_0 \in \mathbb{Q}(\beta_0)$.
We may assume that $\alpha_0 > 1$ as both $\alpha_0$ and $-\alpha_0$ are
    in $\mathbb{Q}(\beta_0)$ and $\beta_0 > 0$.
Then for all Galois actions $\sigma$ on $\mathbb{Q}(\beta_0)$
    we have $\sigma(\beta_0) = \sigma(\alpha_0^p)$.
This gives that $|\alpha_0| > 1$ and all of $\alpha_0$'s conjugates are
    strictly less than $1$ in absolute value.
That is, $\alpha_0$ is either a Pisot number or the negative of a Pisot number.
Since $\alpha_0 > 1$, this gives that $\alpha_0$ is a Pisot number.
This contradicts $n$ being the minimal such divisor of $k$ with this property
    (as we could have taken $n/p$ instead).
Hence there is no $\alpha_0 \in \mathbb{Q}(\beta_0)$ with
    $\beta_0 = \alpha_0^p$, and the result follows from Theorem~\ref{thm:pisot-entropy}.
\end{proof}

\begin{rmk}
The above result is not restricted to Pisot numbers. An equivalent result could have been done for many other
families of algebraic integers. The main property is that the family behaves relatively well under
    powers. For example, a power of a Pisot number is Pisot. Some examples might include
\begin{itemize}
\item Salem numbers (where all conjugates are less than or equal
    to one in absolute value, with at least
    one conjugates equal);
\item Generalized Pisot numbers (where there are $m$ real
      conjugates greater than~$1$ for a fixed $m$ and all other conjugates strictly less
      than $1$ in absolute value);
\end{itemize}
\end{rmk}

Using a similar technique we also show
\begin{cor}
Let $k \geq 2$, $\beta \in (1,2)$ and $\deg(\beta) = r$.
Assume further that $\deg(\beta^{1/k}) > r$.
Then $H_{\beta^{1/k}} \geq 1$ and $\dim(\mu_{\beta^{1/k}}) = 1$.
\end{cor}

\begin{proof}
As before, let $\alpha = \beta^{1/k}$.
Choose $n | k$ such that $\deg(\alpha^n) = r$ and for all $n' | n$, $n' < n$
    we have that $\deg(\alpha^{n'}) > r$.
Let $\beta_0 = \alpha^n$.
We claim that $\beta_0$ satisfies the conditions of Theorem
    \ref{thm:pisot-entropy}, from which the result will follow.

Let $p | n$ and assume that $\beta_0 = \alpha_0^p$ for some
    $\alpha_0 \in \mathbb{Q}(\beta_0)$.
We have that $\deg(\alpha_0) \leq \deg(\mathbb{Q}(\beta_0)) = r$.
As $\beta_0 = \alpha_0^p \in \mathbb{Q}(\alpha_0)$ we get that
    $\deg(\alpha_0) \geq r$.
Hence $\deg(\alpha_0) = r$.
This contradicts $n$ being a minimal such divisor of $k$ with the property that $\deg(\alpha^n) = r$,
    as $\deg(\alpha^{n/p}) = \deg(\alpha_0) = r$.
Hence $n/p$ would also has this property.
Therefore, there is no $\alpha_0 \in \mathbb{Q}(\beta_0)$ with
    $\beta_0 = \alpha_0^p$.
Hence the result follows from Theorem \ref{thm:pisot-entropy}.
\end{proof}

Denote by $\E_n(x;\be)$ the set of all 0-1 words of length~$n$ which may act
as prefixes of $\be$-expansions of $x$.
The following result is an extension of \cite[Lemma~3]{HS2010} to all $\be$.

\begin{lemma}\label{lem:mge}
Suppose there exists $\la\in(1,2)$ and $C>0$ such that $\#\E_n(x;\be)\le C\la^n$ for all $x\in I_\be$.
Then
\[
H_\be\ge \log_\be\frac2{\la}.
\]
\end{lemma}
\begin{proof}
Let $a_1a_2\ldots\in\{0,1\}^{\BN}$ and denote $x_n=\sum_{j=1}^n a_j\be^{-j}$.
Then by \cite[Lemma~4]{Lalley98},
\begin{equation}\label{eq:lalley}
\liminf_{k\to\infty}\left(\frac{p_{nk}(x_{nk})}{2^{nk}}\right)^{1/k}\ge \be^{-H_n(\be)}.
\end{equation}
By (\ref{eq:pnx}) and the definition of $\E_m(x;\be)$, we have
$p_m(x_m)\le \#\E_m(x_m;\be)$, whence $p_{nk}(x_{nk})\le C\la^{nk}$. Hence
by (\ref{eq:lalley}),
\[
\be^{-H_n(\be)}\le \frac{\la^n}{2^n},
\]
whence
\[
\be^{-\frac1n H_n(\be)}\le \frac{\la}2,
\]
which implies, as $n\to\infty$,
\[
\be^{-H_\be}\le \frac{\la}2.
\]
\end{proof}

\section{The algorithm}
\label{sec:alg}

%Let $\beta \in [1,2]$.
This section is concerned with a computer-assisted proof of Theorem~\ref{thm:082}.
First, we need to introduce some useful notation.

For a sequence $(a_1, a_2, \dots, a_n) \in \{0,1\}^n$ we define
    \[
    (a_1, a_2, \dots, a_n)_L = \sum_{i=1}^n a_i \beta^{-i}
    \]
    and
    \[
    (a_1, a_2, \dots, a_n)_U = \sum_{i=1}^n a_i \beta^{-i}
                               + \sum_{i=n+1}^\infty \beta^{-i}.
    \]
We then see that if a $\beta$-expansion of $x$ begins with
    $a_1 a_2 \dots a_n$, then we necessarily have
    \[
    x \in [ (a_1, a_2, \dots, a_n)_L, (a_1, a_2, \dots, a_n)_U].
    \]
We define \[ m_{n}(x,\beta) = \# \{(a_1, \dots, a_n) \mid
    x \in [ (a_1, a_2, \dots, a_n)_L, (a_1, a_2, \dots, a_n)_U]\}\] and
     $m_n(\beta) = \sup_{x\in I_\be} m_n(x, \beta)$.
Lemma~\ref{lem:mge} implies the following claim.

\begin{prop}
For all $n$ and all $\beta \in (1,2)$  we have
\[
  H_\beta \geq \log \frac{2}{m_n(\beta)^{1/n}}.
\]
\end{prop}

A useful fact is that for fixed $n$, the function $m_n(\beta)$ is piecewise constant.
This leads us to an algorithm for computing a lower bound for $H_\beta$.
Note that by Proposition~\ref{prop:Hbetak} for $k=2$ and subsequent
Corollary~\ref{cor:1526}, it suffices to prove Theorem~\ref{thm:main} for
$\be\in(1.526, 2)$.

Our algorithm is as follows:

\begin{enumerate}
\item
Find intervals $I_i$ such that $(1.526, 2) = \bigcup_i I_i$  and such that
    $m_n(\beta)$ is constant on $\mathrm{int}(I_i)$.
    The endpoints of such intervals are called {\em transition points}.
\label{it:1}
\item For all intervals $I_i$ as above and any $\beta \in \mathrm{int}(I_i)$,
    compute $m_n(\beta)$.
\label{it:2}
\item For all intervals $I_i$ as above and $\beta$ an endpoint of $I_i$,
    compute $m_n(\beta)$.
\label{it:3}
\end{enumerate}
Details of this algorithm, along with examples, can be found in \cite{HS2010}.

We then proceed to increase $n$ until such time we have
    $H_\beta > 0.82$ for all $\beta \in (1.526, 2)$.
We eventually needed to search up to degree~18 to do this.
That is not to say that all ranges in $(1.526, 2)$ needed to be searched
    that far.
For example, after degree~4, we no longer need to search between
    1.86 and 1.88.
More striking, after degree~9, we no longer need to search between
    1.526 and 1.7.
All values in this range are already proven to have $H_\beta > 0.82$.
That is, for fixed $n$ we can subdivide $(1.526, 2) = S \cup T$ where we know
    $H_\beta > 0.82$ on $S$ and do not know this on $T$.
Then when working with $n+1$ we only need to find $I_i$ such that  $T \subset \bigcup_i I_i$,
    not all of $(1.526, 2)$.
This reduces the number of tests that need to be done each step.

At high degrees we can optimize even more.
For example, after degree~13 the only intervals not proven to have
    $H_\beta> 0.82$ are [1.8391, 1.8395] and  [1.9274, 1.9277].
(The intervals are actually tighter than this, but this is good enough
    for the discussion.)
At each step we need to search through the list of all degree~14 strings
    $(a_1, a_2 \dots, a_{14})$ compared to itself $(b_1, b_2 \dots, b_{14})$
    and test if there is a transition point coming from
    $(a_1, a_2 \dots, a_{14})_{U/L} = (b_1, b_2 \dots, b_{14})_{U/L}$.
Because we know we are only looking for transition points in the two intervals
    mentioned above, there are huge numbers of strings that can be eliminated
    from this search based on the first $5$ or $6$ terms.
For example, in the interval
    $[1.8391, 1.8395]$, if $a_1 = \dots = a_5 = 1$ then
    $1.135108333 \leq (a_{1}, \dots, a_{14})_{U/L}$.
Further if $b_1 = b_2 = \dots = b_5 = 0$ then
    $(b_{1}, \dots, b_{14})_{U/L} \leq 0.05655621525$.
Hence we know that for any $(a_1, \dots, a_{14})$ starting with
    $a_1 = \dots = a_5 = 1$ and any
    $(b_1, \dots, b_{14})$ starting with
    $b_1 = \dots = b_5 = 0$ we have no transition points in
    $[1.8391, 1.8395]$.
Hence we can remove a huge set of tests without having to test
    each pair individually.

For a particular $\beta$ we see that
    \[
    I_\be= \bigcup [(a_1, \dots, a_n)_L, (a_1, \dots, a_n)_U],
    \]
    with considerable overlap of these intervals.
The computation of $m_n(\beta)$ is equivalent to finding which of these sections has
    maximal overlap.
This can be done reasonably efficiently by first sorting all of the
    strings based on the lower bounds of their associated intervals,
    $(a_1, \dots, a_n)_L$.
Note that necessarily the strings are also sorted by their upper bounds.

Starting at $x =0$, we increase $x$ to the next section where we either add or
    remove an interval of the form $[(a_1, \dots, a_n)_L, (a_1, \dots, a_n)_U]$.
We then keep track at each new section if a new string starts to overlap
    that section, or a previously overlapping string no longer
    overlaps that that section.
This progresses linearly within the sorted set of strings.
This can be done both for numerical tests, as given in \ref{it:2}
    or symbolic tests given in \ref{it:3}.

\section*{Acknowledgment}
The authors are indebted to P\'eter Varj\'u for suggesting
Proposition~\ref{prop:Hbetak} and for stimulating conversations.
The authors would also like to acknowledge David McKinnon for
    the reference to Theorem~\ref{thm:Lang} given in \cite{Lang93}.

\def\lfhook#1{\setbox0=\hbox{#1}{\ooalign{\hidewidth
  \lower1.5ex\hbox{'}\hidewidth\crcr\unhbox0}}}

\end{document}